\documentclass[a4paper,12pt]{article}
\usepackage[left=2cm,right=2cm, top=2cm,bottom=3cm,bindingoffset=0cm]{geometry}
\usepackage{cmap}
\usepackage[T2A]{fontenc}
\usepackage[cp1251]{inputenc}
\usepackage{verbatim}
\usepackage{amsmath}
\usepackage{amsthm}
\usepackage{amssymb}
\usepackage{delarray}
\usepackage{cite}
\usepackage{hyperref}
\usepackage{mathrsfs}
\usepackage{tikz}
\usetikzlibrary{patterns}
\usepackage{caption}
\newcommand{\la}{\lambda}

\newcommand{\gt}{\textgreater}

\DeclareCaptionLabelSeparator{dot}{. }
\theoremstyle{plain}

\numberwithin{equation}{section}

\newtheorem{thm}{Theorem}[section]
\newtheorem{lem}[thm]{Lemma}
\newtheorem{prop}[thm]{Proposition}
\newtheorem{cor}[thm]{Corollary}

\theoremstyle{definition}

\newtheorem{alg}[thm]{Algorithm}
\newtheorem{ip}[thm]{Inverse Problem}

\theoremstyle{remark}

 \allowdisplaybreaks

\begin{document}
\begin{center}
{\large\bf Inverse Sturm-Liouville problem with singular potential and spectral\\[0.2cm] parameter in the boundary conditions}
\\[0.2cm]
{\bf Chitorkin E.E., Bondarenko N.P.} \\[0.2cm]
\end{center}

\vspace{0.5cm}

{\bf Abstract.} This paper deals with the Sturm-Liouville problem that feature distribution potential, polynomial dependence on the spectral parameter in the first boundary condition, and analytical dependence, in the second one. We study an inverse spectral problem that consists in the recovery of the potential and the polynomials from some part of the spectrum. We for the first time prove local solvability and stability for this type of inverse problems. Furthermore, the necessary and sufficient conditions on the given subspectrum for the uniqueness of solution are found, and a reconstruction procedure is developed. Our main results can be applied to a variety of partial inverse problems. This is illustrated by an example of the Hochstadt-Lieberman-type problem with polynomial dependence on the spectral parameter in the both boundary conditions.

\medskip

{\bf Keywords:} inverse spectral problems; Sturm-Liouville operator; polynomials in the boundary conditions; entire functions in the boundary conditions; singular potential; local solvability; stability; Hochstadt-Lieberman problem.

\medskip

{\bf AMS Mathematics Subject Classification (2020):} 34A55 34B07 34B09 34B24 34L40    

\vspace{1cm}

\section{Introduction} \label{sec:intro}

In this paper, we consider the following boundary value problem $L$:
\begin{gather} \label{eqv}
-y'' + q(x) y = \lambda y, \quad x \in (0, \pi), \\ \label{bc}
p_1(\la)y^{[1]}(0) + p_2(\la)y(0) = 0, \quad f_1(\la)y^{[1]}(\pi) + f_2(\la)y(\pi) = 0,
\end{gather}
where $q(x)$ is a complex-valued distribution potential of the class $W_2^{-1}(0,\pi)$, that is, $q = \sigma'$, $\sigma \in L_2(0,\pi)$, $\la$ is the spectral parameter, $p_1(\la)$ and $p_2(\la)$ are relatively prime polynomials, $f_1(\la)$ and $f_2(\la)$ are  arbitrary analytical in the $\la$-plane functions, $y^{[1]} := y' - \sigma(x) y$ is the so-called quasi-derivative. 

We study the inverse problem that consists in the recovery of $\sigma(x)$, $p_1(\lambda)$, and $p_2(\lambda)$ from some part of the spectrum which can be called  a subspectrum. This paper aims to obtain the necessary and sufficient conditions for the uniqueness of solution and to prove local solvability and stability for the inverse problem.


Inverse problems of spectral analysis involve reconstructing a particular differential operator, such as the Sturm-Liouville operator, using known information about its spectrum. These problems have practical uses across various scientific and engineering disciplines, including quantum mechanics, vibration theory, geophysics, chemistry, electronics. Inverse spectral theory for the Sturm-Liouville operators with constant coefficients in their boundary conditions has been studied fairly completely (refer to the monographs \cite{Mar77, Lev84, FY01, Krav20}). 
Inverse Sturm-Liouville problems featuring polynomials of the spectral parameter within the boundary conditions also have garnered attention in the mathematical literature, as evidenced by works \cite{Chu01, BindBr021, BindBr022, ChFr, FrYu, Wang13, YangWei18, Gul19, GulHL, Gul20-ams, Gul23, Chit, ChitBond, ChitBondStab, ChitBondStabMult}. 
Such kind of problems find their applications in modelling vibrations for various types of leads, heat conduction through a liquid-solid interface, some diffusion and electric processes (see \cite{Ful77} and references therein).

There also has been growing interest in exploring a novel category of inverse Sturm-Liouville problems, specifically those containing entire analytic functions of the spectral parameter within one of the boundary conditions (see, e.g., \cite{BondPart, BondStabAn, YangBondXu, Chit, KuzCauchy}). The problems of this type are closely related to the so-called partial inverse problems on intervals and on graphs, which generalize the famous Hochstadt-Lieberman problem \cite{HL}.
In general, partial inverse problems consist in the recovery of the differential expression parameters (e.g., the Sturm-Liouville potential) on a part of an interval or of a geometrical graph from spectral information, while the parameters on the remaining part of the domain are known a priori. It has been noticed that such problems can be reduced to complete inverse problems on the ``unknown'' part of the domain. In this case, one of the boundary conditions contains analytic functions that are constructed by using the parameters from the ``known'' part of the domain.
Basing on these ideas, the unified approach to various types of partial inverse problems has been developed (see the overview \cite{BondPart}). That approach is also applicable to the inverse transmission eigenvalue problem, which has been actively studied in recent yeas (see, e.g, \cite{GP17, BCK20}).
The inverse Sturm-Liouville problem with polynomials of the spectral parameter in one of the boundary conditions and with entire functions in the other one previously has been considered in the only paper \cite{Chit}. However, the results of \cite{Chit} are limited to uniqueness theorems, while the issues of solvability and stability for that type of inverse spectral problems remained open.

This work aims to create the inverse spectral theory for the boundary value problem \eqref{eqv}-\eqref{bc} with \textit{distribution} potential $q$ of the class $W_2^{-1}(0,\pi)$, which is wider than the class $L_2(0,\pi)$ of regular potentials considered in \cite{Chit}. The theory of direct and inverse Sturm-Liouville problems with distribution potentials has been developed in \cite{SavShkal03, Hry03, Sav10} are other studies. A feature of our problem is polynomial dependence on $\la$ in one of the boundary conditions and analytical dependence in the other one. For this type of inverse spectral problems, we for the first time obtain local solvability and stability. Furthermore, we find conditions on the given subspectrum that guarantee the uniqueness of the inverse problem solution. In contrast to \cite{Chit}, we prove not only the sufficiency but also the necessity of our conditions.

For the investigation of the inverse problem, we develop an approach based on the construction of a special vector functional sequence $\{v_n\}_{n \ge 1}$ related to the given data. We prove that the completeness of this system is necessary and sufficient for the uniqueness of the inverse problem solution. Next, we obtain a constructive procedure for recovering $\sigma(x)$, $p_1(\la)$, and $p_2(\la)$. Our algorithm reduces the inverse problem for \eqref{eqv}--\eqref{bc} with entire functions $f_1(\la)$ and $f_2(\la)$ to the reconstruction of the desired parameters from the Weyl function for the Sturm-Liouville problem with polynomials in one of the boundary conditions. The latter problem for the case of distribution potential has been studied in our previous papers \cite{ChitBond, ChitBondStab, ChitBondStabMult}.
To get sufficient conditions for the validity of our constructive procedure, we study the unconditional basicity of the system $\{v_n\}_{n \ge 1}$. We follow the ideas of \cite{BondStabAn}, which have been initially developed for the case of the Dirichlet boundary condition $y(0) = 0$. However, for polynomials of the spectral parameter in the boundary condition, the construction of the system $\{ v_n\}_{n \ge 1}$ and the study of its properties become more technically complicated. Furthermore, we find conditions on the functions $f_1(\la)$, $f_2(\la)$ and on the given subspectrum that guarantee local solvability and stability of the inverse spectral problem.

Our main results can be applied to a variety of partial inverse problems with polynomials in the boundary conditions. As an example, we consider the Hochstadt-Lieberman-type inverse problem with polynomial dependence of the spectral parameter in the both boundary conditions. This problem consists in the recovery of $\sigma(x)$ on a half-interval and of the polynomials in the first boundary condition from the spectrum, while $\sigma(x)$ on the other half-interval and the polynomials in the second boundary condition are known a priori.
It is worth mentioning that problems of this kind were previously investigated in \cite{Wang13, GulHL, Chit}, where uniqueness theorems for them were obtained. Applying our main results, we directly get not only uniqueness but also local solvability and stability for the Hochstadt-Lieberman-type inverse problem. More precisely, we find a maximal number of eigenvalues that can be excluded from the spectrum so that the remaining subspectrum uniquely determines the problem parameters.
In the future, our results can be applied to other types of inverse problems, e.g., to half-inverse problems with polynomial dependence in discontinuity conditions inside the interval (see \cite{Chit, BCNW18, BCW21, DGW23}).
Moreover, as our approach is constructive and the stability is proved, then numerical methods for solving the considered inverse problems can be developed and used in physical applications.

The paper is organized as follows. In Section~\ref{sec:main}, we formulate the main results of this paper. In Section~\ref{sec:prelim}, we provide some preliminaries, which are concerned with inverse Sturm-Liouville problems featuring polynomials in a boundary condition. In Section~\ref{sec:uniq}, the sequence $\{ v_n \}_{n \ge 1}$, which plays a central role in our method, is constructed, the uniqueness theorem for the inverse problem solution is proved, and a reconstruction algorithm is presented.
In Section~\ref{sec:suff}, we prove a theorem on sufficient conditions for the completeness and for the basicity of $\{ v_n \}_{n \ge 1}$.
In Section~\ref{sec:stab}, local solvability and stability of the inverse problem are proved. In Section~\ref{sec:h-l}, we apply our results to the Hochstadt-Lieberman-type problem with polynomial dependence on spectral parameter in the boundary conditions.

Throughout the paper, we use the notations $\rho = \sqrt{\la}$ and $\rho_n = \sqrt{\la_n}$, where the branch of the square roots is such that $\arg \rho, \, \arg \rho_n \in \left[ -\tfrac{\pi}{2}, \tfrac{\pi}{2} \right)$, $\tau = \Im \rho$.
The notation $\varkappa_a(\rho)$ is used for various entire functions in the $\rho$-plane of exponential type not greater than $a$ belonging to $L_2(\mathbb R)$, which are called the Paley-Wiener functions and admit the representation
$$
\varkappa_a(\rho) = \int_{-a}^a \xi(x) e^{i \rho x} \, dx, \quad \xi \in L_2(-a, a).
$$
Further, we need the following estimate for the Paley-Wiener functions:
$$
\varkappa_a(\rho) = o(e^{|\tau|a}), \quad |\rho| \to \infty.
$$

\section{Main results} \label{sec:main}

Consider the boundary value problem $L = L(\sigma, p_1, p_2, f_1, f_2)$ of form \eqref{eqv}-\eqref{bc}. Due to the standard regularization approach of \cite{SavShkal03},
we understand equation \eqref{eqv} in the following equivalent sense:
\begin{equation} \notag
-(y^{[1]})' - \sigma(x)y^{[1]} - \sigma^2(x)y = \la y, \quad x \in (0,\pi),
\end{equation}
where $y, y^{[1]} \in AC[0,\pi]$.
Obviously, the polynomials of the boundary conditions \eqref{bc} can be expressed in the form
$$
p_1(\la) = \sum \limits_{n=0}^{N_1} a_n \la^n, \quad  p_2(\la) = \sum \limits_{n=0}^{N_2} b_n \la^n, \quad N_1, N_2 \ge 0.
$$

Denote $p := \max\{2N_1+1, 2N_2\}$. So, $p$ can be odd (if $2N_1 +1 > 2N_2$) or even (if $2N_1 + 1 < 2N_2$). If $p$ is odd, then $N_1 \ge N_2$. In this case, without loss of generality, we can say that $b_{N_1}$, $b_{N_1-1}$, ..., $b_{N_2+1}$ are zero. So, we suppose that $N_1 = N_2$ and $a_{N_1} = 1$. If $p$ is even, then $N_1 < N_2$. In this case, without loss of generality, we can say that $a_{N_2}$, $a_{N_2-1}$, ..., $a_{N_1+1}$ are zero. So, we suppose that $N_1 = N_2 - 1$ and $b_{N_2} = 1$. Next, for such pairs of polynomials we will write $(p_1(\la), p_2(\la)) \in \mathcal R_p$.

This paper is devoted to investigation of the following problem.

\begin{ip}\label{ip:main}
Given $p$, the functions $f_1(\la)$, $f_2(\la)$ and a subspectrum $\{\la_n\}_{n \ge 1}$ of the problem $L$, find $\sigma(x)$ and the polynomials $p_1(\la)$, $p_2(\la)$.
\end{ip}

Let us briefly describe our method. We introduce the Hilbert space
$$
\mathcal{H}_p = L_2(0, \pi) \oplus L_2(0, \pi) \oplus \underbrace{\mathbb {C} \oplus \dots \oplus \mathbb {C}}_{p}
$$
and reduce Inverse Problem~\ref{ip:main} to the system of equations 
\begin{equation} \label{sys}
 (u, v_n)_{\mathcal{H}_p}=w_n, \quad n \ge 1,
\end{equation}
where $u$ is an unknown vector function in $\mathcal{H}_p$, the vector functions $v_n \in \mathcal H_p$ and the right-hand sides $w_n$ are found by using the given data $f_1(\la)$, $f_2(\la)$ and the subspectrum $\{\la_n\}_{n \ge 1}$.
Depending on the parity of $p$, the sequences $\{ v_n \}_{n \ge 1}$ and $\{ w_n \}_{n \ge 1}$ are constructed in different ways (see Section~\ref{sec:uniq} for details). 
Relying on the system \eqref{sys}, we prove the uniqueness theorem, develop a constructive algorithm, and find conditions of local solvability and stability for Inverse Problem~\ref{ip:main}.

Proceed to the formulations of the main results.
Along with the problem $L$, we consider another problem $\tilde L(\tilde\sigma, \tilde p_1, \tilde p_2, \tilde f_1, \tilde f_2)$ of the same form, but with other coefficients. Next, we assume that, if a symbol $\gamma$ denotes some object related to the problem $L$, then the symbol $\tilde\gamma$ with tilde denotes the analogous object related to the problem $\tilde L$. Throughout this paper, we suppose that $f_1(\la) \equiv \tilde f_1(\la)$, $f_2(\la) \equiv \tilde f_2(\la)$, and $p \equiv \tilde p$. In this case, the following uniqueness theorem holds:

\begin{thm}\label{thm:uniq_ns}
Completeness of the system $\{v_n\}_{n \ge 1}$ is necessary and sufficient for the uniqueness of solution for Inverse Problem~\ref{ip:main}.
\end{thm}

For the case when $\{ v_n \}_{n \ge 1}$ is an unconditional basis in $\mathcal H_p$, we develop Algorithm~\ref{alg:1} for reconstruction of $q(x)$, $p_1(\la)$, and $p_2(\la)$. Specifically, by solving the system \eqref{sys}, we reduce Inverse Problem~\ref{ip:main} to the previously studied inverse Sturm-Liouville problem with polynomials in one of the boundary conditions.

Since the construction of the system $\{ v_n\}_{n \ge 1}$ is complicated, then checking its completeness and basicity is non-trivial. Therefore, our next result is concerned with some simple sufficient conditions for these properties. 

Introduce the classes $\mathcal S$ and $\mathcal A$ for subspectra $\{ \la_n \}_{n \ge 1}$. Namely, we will write that:
\begin{itemize} 
\item $\{ \la_n \}_{n \ge 1} \in \mathcal S$ if $\la_n \ne \la_k$ for any $n \ne k$, that is, the subspectrum $\{ \la_n \}_{n \ge 1}$ is simple;
\item $\{ \la_n \}_{n \ge 1} \in \mathcal A$ if $\la_n = \rho_n^2 \neq 0$ and $\Im \rho_n = O(1)$ for all $n \ge 1$, and also $\left\{ \rho_n^{-1} \right\}_{n  \ge 1} \in l_2$.
\end{itemize}

\begin{thm}\label{thm:basis}
Suppose that $\{\la_n\}_{n \ge 1} \in \mathcal S$ and $f_1(\la_n) \neq 0$ or $f_2(\la_n) \neq 0$ for any $n \ge 1$. Then, the following assertions are valid:
\begin{enumerate}
\item If the system $\{\sin\sqrt{\la_n} t\}_{n > p}$ is complete in $L_2(0, 2\pi)$, then the system $\{ v_n \}_{n \ge 1}$ is complete in $\mathcal{H}_p$. Thus, the solution of Inverse Problem~\ref{ip:main} is unique.

\item If the system $\{\sin\sqrt{\la_n} t\}_{n > p}$ is a Riesz basis in $L_2(0, 2\pi)$ and $\{ \la_n \}_{n \ge 1} \in \mathcal A$, then the system $\{ v_n \}_{n \ge 1}$ is an unconditional basis in $\mathcal{H}_p$. Thus, the solution of Inverse Problem~\ref{ip:main} can be found by Algorithm~\ref{alg:1}.
\end{enumerate}
\end{thm}

Out next important result is the theorem on local solvability and stability of Inverse Problem~\ref{ip:main}:

\begin{thm}\label{thm:stab}
Suppose that $f_1(\la)$ and $f_2(\la)$ are entire functions, $(\tilde p_1(\la), \tilde p_2(\la)) \in \mathcal R_p$ are polynomials, $\tilde \sigma(x) \in L_2(0, \pi)$ is a complex-valued function, and $\{\tilde\la_n\}_{n \ge 1} \in \mathcal S$ is a fixed subspectrum of the problem $\tilde L$. Assume that:
\begin{enumerate}
\item $\{\tilde v_n\}_{n \ge 1}$ is an unconditional basis in $\mathcal{H}_p$.
\item $\{ \tilde\la_n \}_{n \ge 1} \in \mathcal A$.
\item The following estimates hold:
\begin{equation} \label{estf}
\arraycolsep=1.4pt\def\arraystretch{1.8}
\begin{array}{c}
|f_1(\rho^2)| \le a_1|\tilde\rho_n|^{\alpha_n}, \quad |f_2(\rho^2)| \le a_1|\tilde\rho_n|^{\alpha_n+1}, \quad |\rho-\tilde\rho_n| \le a_2, \\
|f_1(\tilde\la_n)|^2 + \dfrac{1}{|\tilde\la_n|}|f_2(\tilde\la_n)|^2 \ge a_3|\tilde\la_n|^{\alpha_n},
\end{array}
\end{equation}
where $a_1, a_2, a_3$ are some positive real numbers, $\{\alpha_n\}$ is some real sequence.
\end{enumerate}

Then, there exists $\varepsilon > 0$ such that, for any complex numbers $\{\la_n\}_{n \ge 1}$, satisfying the condition
\begin{gather}\label{est_omega}
\Omega = \Bigg(\sum\limits_{n=1}^{\infty}|\tilde\rho_n - \rho_n|^2\Bigg)^{\frac{1}{2}} \le \varepsilon,
\end{gather}
where $\rho_n^2 = \la_n$, $\tilde \rho_n^2 = \tilde\lambda_n$, there exist a complex-valued function $\sigma(x) \in L_2(0, \pi)$ and polynomials $(p_1(\la), p_2(\la)) \in \mathcal R_p$ such that $\{\la_n\}_{n \ge 1}$ is a subspectrum of the problem $L = L(\sigma, p_1, p_2, f_1, f_2)$. Moreover,
\begin{gather} \label{coef_estimates}
\|\sigma(x) - \tilde\sigma(x)\|_{L_2(0, \pi)} \le C\Omega, \quad |a_j - \tilde a_j| \le C\Omega, \quad |b_j - \tilde b_j| \le C\Omega,
\end{gather}
where $C$ depends only on $\tilde L$ and $\{\tilde\la_n\}_{n \ge 1}$.
\end{thm}

The simplicity condition for the subspectrum $\{ \la_n \}$ in Theorems~\ref{thm:basis} and~\ref{thm:stab} can be removed by using the technique of \cite{BondStabAn}. However, the case of multiple eigenvalues is more technically complicated. It requires additional notations and constructions. Thus, for clear presentation of our main ideas, we confine ourselves to the case of simple eigenvalues.

The hypotheses of Theorems~\ref{thm:basis} and~\ref{thm:stab} naturally hold for applications to partial inverse problems. As an example, we consider the Hochstadt-Lieberman-type inverse problem for equation \eqref{eqv} with polynomials in the both boundary conditions in Section~\ref{sec:h-l}. We show that, relying on Theorem~\ref{thm:basis}, one can easily figure out the maximal number of eigenvalues that can be excluded in order to keep the unique solvability of the Hochstadt-Lieberman-type problem. Moreover, Theorem~\ref{thm:stab} readily implies the local solvability and stability of the corresponding inverse problem. In particular, this yields the minimality of the given subspectrum.

\section{Preliminaries} \label{sec:prelim}

In this section, we introduce some notations and formulate the known results, which are used in proof of Theorem~\ref{thm:uniq_ns}. Our preliminaries are based on the previous studies \cite{FrYu, Chit, ChitBond, ChitBondStabMult}.

Let us define $S(x, \la)$ and $C(x, \la)$ as the solutions of equation~\eqref{eqv} under the initial conditions:
\begin{equation} \label{initphi}
S(0, \la) = 0, \quad S^{[1]}(0, \la) = 1, \quad C(0, \la) = 1, \quad C^{[1]}(0, \la) = 0.
\end{equation}
Clearly, the functions $S(x, \la)$, $S^{[1]}(x, \la)$, $C(x, \la)$ and $C^{[1]}(x, \la)$ are entire analytic in $\la$ for each fixed $x \in [0,\pi]$.

The spectrum of the problems $L$ coincide with the zeros of the characteristic function
\begin{gather}\label{charfun}
\Delta(\la) = f_1(\la)\Delta_1(\la) + f_2(\la)\Delta_0(\la),
\end{gather}
where 
$$
\Delta_j(\la) := p_1(\la)C^{[j]}(\pi, \la) - p_2(\la)S^{[j]}(\pi, \la), \quad j = 0, 1, 
$$
are the characteristic functions of the boundary value problems $L_j$ for equation~\eqref{eqv} with the boundary conditions
\begin{equation}
\left\{ \begin{aligned} 
p_1(\la)y^{[1]}(0) + p_2(\la)y(0) &= 0, \\
y^{[j]}(\pi) &= 0.
\end{aligned} \right.
\end{equation}

An important property of $\Delta_0(\la)$ and $\Delta_1(\la)$ is reflected in the following statement, which is analogous to Lemma 1 in \cite{FrYu}:
\begin{prop} \label{prop:zeros}
	$\Delta_0(\la)$ and $\Delta_1(\la)$ have no common zeros.
\end{prop}

Define the Weyl function $M(\la) := \dfrac{\Delta_0(\la)}{\Delta_1(\la)}$ of the boundary value problem $L_1$ 
and consider the following auxiliary inverse problem.

\begin{ip} \label{ip:Weyl}
	Given the Weyl function $M(\la)$, find $q(x)$, $p_1(\la)$, and $p_2(\la)$.
\end{ip}

The uniqueness theorem for Inverse Problem~\ref{ip:Weyl} was proved in \cite{ChitBond}.
\begin{prop}[Theorem 2.3 from \cite{ChitBond}] \label{prop:weyl_uniq}
	If $M(\la) \equiv \tilde M(\la)$, then $\sigma(x) = \tilde\sigma(x)$ a.e. on $x \in (0, \pi)$ and $p_j(\la) = \tilde p_j(\la)$, $j = 1, 2$.
\end{prop}

A constructive procedure for solving Inverse Problem~\ref{ip:Weyl} is presented in \cite[Algorithm~6.1]{ChitBond}. That procedure is based on the reduction of the inverse spectral problem to a linear equation in the space of bounded infinite sequences. Using the solution of that equation, one can recover $\sigma(x)$, $p_1(\la)$, and $p_2(\la)$ by using the reconstruction formulas from \cite{ChitBond}.

Using the technique from \cite{Chit}, we obtain the following representations for $\Delta_j(\la)$, $j=0, 1$:

\begin{lem}\label{lem:reprsdeltas}
The following representations hold:
\begin{gather} 
\label{Delta1}
\Delta_1(\la) = \left\{ \begin{aligned} 
-\la^{N_1+1}\dfrac{\sin\rho\pi}{\rho} + \la^{N_1+1}\int_0^{\pi}  \mathcal{J}(t) \dfrac{\sin\rho t}{\rho} \, dt + \sum\limits_{n=0}^{N_1} A_{2n+1}\la^n, \quad N_1 \ge N_2, \\
-\la^{N_2}\cos\rho\pi + \la^{N_2}\int_0^{\pi}  \mathcal{J}(t) \cos\rho t \, dt + \sum\limits_{n=0}^{N_2-1} A_{2n+1}\la^n, \quad N_1 < N_2,
\end{aligned} \right. \\ \label{Delta0}
\Delta_0(\la) = \left\{ \begin{aligned} 
\la^{N_1}\cos\rho\pi + \la^{N_1}\int_0^{\pi}  \mathcal{G}(t) \cos\rho t \, dt + \sum\limits_{n=0}^{N_1-1} A_{2n+2}\la^n, \quad N_1 \ge N_2, \\
-\la^{N_2}\dfrac{\sin\rho\pi}{\rho} + \la^{N_2}\int_0^{\pi}  \mathcal{G}(t) \dfrac{\sin\rho t}{\rho} \, dt + \sum\limits_{n=0}^{N_2-1} A_{2n+2}\la^n, \quad N_1 < N_2,
\end{aligned} \right.
\end{gather}
where $\mathcal{J}(t), \mathcal{G}(t) \in L_2(0, \pi)$ and $A_n \in \mathbb {C}$.
\end{lem}

We will call the set $\{\mathcal{J}(t), \mathcal{G}(t), A_1, \dots, A_p\}$ the {\it generalized Cauchy data}. This notion has been introduced in \cite{ChitBondStabMult} for the case of polynomials in the boundary conditions. The analogous Cauchy data for inverse problems with constant coefficients were used in \cite{BondStabAn, BondPart, RS92} and other studies.

Note that the Weyl function $M(\la)$ and the generalized Cauchy data mutually uniquely determine each other. Therefore, Inverse Problem~\ref{ip:Weyl} is equivalent to the following problem:

\begin{ip} \label{ip:Cauchy}
Given the generalized Cauchy data $\{ \mathcal{J}(t), \mathcal{G}(t), A_1, \dots, A_p\}$, find $\sigma(x)$, $p_1(\la)$, and $p_2(\la)$.
\end{ip}

To prove Theorem~\ref{thm:uniq_ns}, we need the following result for Inverse Problem~\ref{ip:Cauchy}:
\begin{prop}[Theorem~2.7 from \cite{ChitBondStabMult}] \label{cauchy_thm}
	Let $\tilde\sigma \in L_2(0, \pi)$, $(\tilde p_1, \tilde p_2) \in \mathcal R_p$ and let $\{\tilde{\mathcal{J}}(t), \tilde{\mathcal{G}}(t), \tilde A_1, \dots, \tilde A_{p}\}$ be the generalized Cauchy data for the problem $\tilde L = L_1(\tilde \sigma, \tilde p_1, \tilde p_2)$. Then, there exists $\varepsilon > 0$ such that, for any $\mathcal{J}, \mathcal{G} \in L_2(0, \pi)$, $A_1, \dots, A_{p} \in \mathbb C$ satisfying the condition
	\begin{align} \label{difCauchy}
	\delta := \max\{&\|\mathcal{J}(t)-\tilde{\mathcal{J}}(t)\|_{L_2(0, \pi)}, \|\mathcal{G}(t)-\tilde{\mathcal{G}}(t)\|_{L_2(0, \pi)}, |A_1 - \tilde A_1|, \dots, |A_{p} - \tilde A_{p}|\} \le \varepsilon,
	\end{align}
	there exist a complex-valued function $\sigma \in L_2(0, \pi)$ and polynomials $(p_1, p_2)\in\mathcal R_p$, which are the solution of Inverse Problem~\ref{ip:Cauchy} by the generalized Cauchy data $\{\mathcal{J}(t), \mathcal{G}(t), A_1, \dots, A_{p}\}$. Moreover,
	\begin{equation} \notag
	\|\sigma - \tilde\sigma\|_{L_2(0, \pi)} \le C\delta, \quad |a_n - \tilde a_n| \le C\delta, \quad |b_n - \tilde b_n| \le C\delta, 
	\end{equation}
	where the constant $C$ depends only on $\tilde L_1$.
\end{prop}

\section{Uniqueness and algorithm} \label{sec:uniq}

In this section, first, we construct the sequence of vector functions $\{ v_n \}_{n \ge 1}$, which plays a crucial role in the analysis of Inverse Problem~\ref{ip:main}. Second, Theorem~\ref{thm:uniq} on the uniqueness of solution for Inverse Problem~\ref{ip:main} is proved by considering separately necessity and sufficiency. It is remarkable that, for proving the sufficiency of the uniqueness conditions, we rely on the local solvability of Inverse Problem~\ref{ip:Cauchy} by the generalized Cauchy data. Third, we develop a constructive algorithm for the recovery of $\sigma(x)$, $p_1(\la)$, and $p_2(\la)$ from the given subspectrum $\{ \la_n \}_{n \ge 1}$. Our main idea consists in the reduction of Inverse Problem~\ref{ip:main} with analytical functions in the boundary condition to Inverse Problem~\ref{ip:Weyl} by the Weyl function (or the equivalent Inverse Problem~\ref{ip:Cauchy} by the generalized Cauchy data).

Recall that $p := \max\{2N_1+1, 2N_2\}$ and consider the Hilbert space
$$
\mathcal{H}_p = L_2(0, \pi) \oplus L_2(0, \pi) \oplus \underbrace{\mathbb {C} \oplus \dots \oplus \mathbb {C}}_{p}
$$
of elements $[H_1(t), H_2(t), h_1, h_2, \dots, h_p]$,
where $H_i(t) \in L_2(0, \pi)$, $h_j \in \mathbb {C}$, $i=1, 2$, $j=\overline{1, p}$. The scalar product and the norm in $\mathcal H_p$ are defined as follows:
$$
(g, h)_{\mathcal{H}_p} = \int_0^{\pi}  \big(\overline{G_1(t)}H_1(t) + \overline{G_2(t)}H_2(t) \big) \, dt + \sum\limits_{n=1}^p{\overline{h_n}g_n}, \quad
\|h\|_{\mathcal{H}_p} = \sqrt{(h, h)_{\mathcal{H}_p}},
$$
where $g, h \in \mathcal{H}_p$.

Introduce the vector function
\begin{equation} \label{defu}
    u(t) = [\overline {{\mathcal {J}}(t)}, \overline {{\mathcal {G}}(t)}, \overline A_{1}, \dots , \overline A_{p}].
\end{equation}

According to Lemma~\ref{lem:reprsdeltas}, consider the two cases.

\medskip

\textbf{\textit{\underline{Case 1:}}} $N_1 \ge N_2$, so $p = 2N_1 + 1$ is odd.

Define the vector function
\begin{gather} \label{defv1}
    v(t, \la) = [f_1(\la)\rho^{p}\sin\rho t, f_2(\la)\rho^{p-1}\cos\rho t, f_1(\la), f_2(\la), \dots , f_1(\la)\rho^{p-3}, f_2(\la)\rho^{p-3}, f_1(\la)\rho^{p-1}],
\end{gather}
and find the scalar product in $\mathcal {H}_p$:
\begin{gather*}
    (u(t), v(t, \la))_{\mathcal H_p} = f_1(\la)\rho^{p} \int_0^{\pi} {\mathcal {J}}(t) \sin\rho t \, dt + f_2(\la)\rho^{p-1} \int_0^{\pi} {\mathcal {G}}(t) \cos\rho t \, dt \\
    + f_1(\la)\sum\limits_{n=0}^{N_1}A_{2n+1}\rho^{2n} + f_2(\la)\sum\limits_{n=0}^{N_1-1}A_{2n+2}\rho^{2n}.
\end{gather*}

According to \eqref{charfun}, \eqref{Delta1}, and \eqref{Delta0}, we can conclude that
\begin{gather} \label{uvw1}
    (u(t), v(t, \la))_{\mathcal H_p} = \Delta(\la) + w(\la),
\end{gather}
where 
\begin{gather} \label{defw1}
    w(\la) = f_1(\la)\rho^p\sin\rho\pi - f_2(\la)\rho^{p-1}\cos\rho\pi.
\end{gather}

\textbf{\textit{\underline{Case 2:}}} $N_2 \gt N_1$, so $p = 2 N_2$ is even.

Define the vector function
\begin{gather} \label{defv2}
    v(t, \la) = [f_1(\la)\rho^{p}\cos\rho t, f_2(\la)\rho^{p-1}\sin\rho t, f_1(\la), f_2(\la), \dots , f_1(\la)\rho^{p-2}, f_2(\la)\rho^{p - 2}],
\end{gather}
Using \eqref{charfun}, \eqref{Delta1}, and \eqref{Delta0}, we conclude that
\begin{gather} \label{uvw2}
    (u(t), v(t, \la))_{\mathcal H_p} = \Delta(\la) + w(\la), 
\end{gather}
where 
\begin{equation} \label{defw2}
w(\la) = f_1(\la)\rho^p\cos\rho\pi + f_2(\la)\rho^{p-1}\sin\rho\pi.
\end{equation}

Introduce the notation
$$
   f^{\langle n \rangle}(\la) = \frac{d^n f}{d\la^n}, \quad n \ge 0.
$$

Let $\{\la_n\}_{n\ge1}$ be a subspectrum of problem $L$. Let set $I := \{n \in \mathbb {N}: \la_{n+1} \neq \la_n \}$ be a set of distinct eigenvalues' indexes and $m_k$ be a multiplicity of the eigenvalue $\la_k$. Without loss of generality, the eigenvalues can be numbered so that $\la_{k-1} \neq \la_k = \la_{k+1} = \dots = \la_{k+m_k-1} \neq \la_{k+m_k}$.

Since $\la_k$ is the zero of $\Delta(\la)$ of multiplicity at least $m_k$, we have
$$
\Delta^{\langle n \rangle}(\la_k) = 0, \quad k \in I, \quad n = \overline {0, m_k - 1}.
$$
Consequently, it follows from \eqref{uvw1} and \eqref{uvw2} that
\begin{gather*}
     (u(t), v^{\langle n \rangle}(t, \la_k))_{\mathcal H_p} = w^{\langle n \rangle}(\la_k), \quad k \in I, \quad n = \overline {0, m_k - 1},
\end{gather*}
in the both cases.

Put 
\begin{equation} \label{defvw}
v_{k + n}(t) = v^{\langle n \rangle}(t, \la_k), \quad w_{k + n}(t) = w^{\langle n \rangle}(t, \la_k), \quad k \in I, \quad n = \overline {0, m_k - 1}.
\end{equation}

Thus, we have defined the sequence $\{ v_n \}_{n \ge 1}$ in ${\cal H}_p$ and the sequence of complex numbers $\{ w_n \}_{n \ge 1}$.
Using \eqref{uvw1} and \eqref{uvw2}, we arrive at the relation
\begin{equation} \label{scal}
    (u, v_n) = w_n, \quad n \ge 1,
\end{equation}
which plays a crucial role in investigation of the inverse problem. Here $\{ v_n \}_{n \ge 1}$ and $\{ w_n \}_{n \ge 1}$ are constructed by using the known data of Inverse Problem~\ref{ip:main}, while $u \in {\cal H}_p$ is related to the unknown $\sigma(x)$ and the polynomials $p_1(\la)$, $p_2(\la)$.

Now we are ready to prove Theorem~\ref{thm:uniq_ns}. We will prove it by parts by formulating additional lemmas.

\begin{lem}[Sufficiency] \label{thm:uniq}
Let $\{ \la_n \}_{n \ge 1}$ and $\{ \tilde \la_n \}_{n \ge 1}$ be subspectra of the problems $L$ and $\tilde L$, respectively.
Suppose that the sequence $\{ v_n \}_{n \ge 1}$ constructed for the problem $L$ and its subspectrum $\{ \la_n \}_{n \ge 1}$ by formulas \eqref{defv1}, \eqref{defv2}, and \eqref{defvw} is complete in $\mathcal{H}_p$, 
and let $p = \tilde p$, $f_j(\la) \equiv \tilde f_j(\la)$, $j = 1,2$, $\la_n = \tilde \la_n$, $n \ge 1$. 
Then $\sigma = \tilde\sigma$ in $L_2(0, \pi)$ and $p_j(\la) \equiv \tilde p_j(\la)$, $j = 1,2$.
\end{lem}

\begin{proof}
Suppose that two boundary value problems $L$ and $\tilde L$ of form \eqref{eqv}-\eqref{bc} and their subspectra $\{ \la_n \}_{n \ge 1}$ and $\{ \tilde \la_n \}_{n \ge 1}$ fulfill the conditions of the lemma.
By construction, we have $v_n = \tilde v_n$ in the Hilbert space ${\mathcal H}_p$ and $w_n = \tilde w_n$, $n \ge 1$. 
Then, for the problem $\tilde L$, we get $(\tilde u, v_n)_{\mathcal{H}_p} = w_n$, $n \ge 1$.
Therefore $(u - \tilde u, v_n)_{\mathcal{H}_p} = 0$, $n \ge 1$.
Due to the completeness of the sequence $\{ v_n \}_{n \ge 1}$ in ${\mathcal H}_p$, this implies $u = \tilde u$ in ${\mathcal{H}_p}$. 
Hence, the generalized Cauchy data of the problems $L$ and $\tilde L$ coincide. Consequently, from \eqref{Delta1}-\eqref{Delta0}, we get that $\Delta_j(\la) \equiv \tilde {\Delta}_j(\la)$, $j = 0, 1$.
Then, $M(\la) \equiv \tilde M(\la)$ and, due the Proposition~\ref{prop:weyl_uniq}, $\sigma = \tilde\sigma$ in $L_2(0, \pi)$ and $p_j(\la) \equiv \tilde p_j(\la)$, $j= 1, 2$.
\end{proof}

\begin{lem}[Necessity] \label{thm:uniq2}
Let $\{\la_n\}_{n \ge 1}$ be a subspectrum of the problem $L$. Suppose that the sequence $\{ v_n \}_{n \ge 1}$ is not complete in $\mathcal{H}_p$. Then, there exist a complex-valued function $\tilde \sigma(x) \in L_2(0, \pi)$ and polynomials $(\tilde p_1(\la), \tilde p_2(\la)) \in \mathcal R_p$ such that $\{\la_n\}_{n \ge 1}$ is a subspectrum of problem $\tilde L$ and $(\tilde \sigma, \tilde p_1, \tilde p_2) \neq (\sigma, p_1, p_2)$.
\end{lem}

\begin{proof}
As the system $\{ v_n \}_{n \ge 1}$ is not complete in $\mathcal{H}_p$, then there exists a non-zero element $\hat u \in \mathcal{H}_p$ such that
\begin{gather}\label{uniq_sys1}
(\hat u, v_n)_{\mathcal{H}_p} = 0, \quad n\ge 1.
\end{gather}
As equations \eqref{uniq_sys1} are linear by $\hat u$, then we can choose $\hat u$ arbitrarily small by the norm. So, let $\|\hat u\|_{\mathcal{H}_p} \le \varepsilon$, where  $\varepsilon$ is taken from Proposition~\ref{cauchy_thm}. Suppose that
\begin{gather*}
u = [\overline{\mathcal{J}(t)}, \overline{\mathcal{G}(t)}, \overline{A_1}, \dots, \overline{A_p}], \\
\tilde u := u + \hat u = [\overline{\tilde{\mathcal{J}}(t)}, \overline{\tilde{\mathcal{G}}(t)}, \overline{\tilde{A}_1}, \dots, \overline{\tilde{A}_{p}}],
\end{gather*}
where $\tilde u \neq u$. Due to Proposition~\ref{cauchy_thm}, there exist a function $\tilde \sigma(x)$ and a pair of polynomials $(\tilde{p}_1(\la), \tilde{p}_2(\la)) \in \mathcal{R}_p$ such that $\{\tilde{ \mathcal{J}}(t), \tilde{\mathcal{G}}(t), \tilde {A}_1, \dots, \tilde {A}_{p}\}$ are the generalized Cauchy data for $\tilde L = L(\tilde \sigma, \tilde p_1, \tilde p_2, f_1, f_2)$.

Next, define the functions $\tilde\Delta_0(\la)$ and $\tilde\Delta_1(\la)$ by formulas \eqref{Delta1}--\eqref{Delta0}, depending on the parity of $p$, and put $\tilde\Delta(\la) = f_1(\la)\tilde\Delta_1(\la) + f_2(\la)\tilde\Delta_0(\la)$. It can be shown, that $\tilde\Delta(\la)$ is the characteristic function of problem $\tilde L$. From \eqref{uniq_sys1} and \eqref{scal} we get $(\tilde u, v_n)_{\mathcal{H}_p} = w_n$, $n \ge 1$. So, $\tilde\Delta(\la)$ has zeros $\{\la_n\}_{n \ge 1}$ and, consequently, $\{\la_n\}_{n \ge 1}$ is the subspectrum of the problem $\tilde L$ with $(\tilde \sigma, \tilde p_1, \tilde p_2) \neq (\sigma, p_1, p_2)$.
\end{proof}

Lemmas \ref{thm:uniq} and \ref{thm:uniq2} together yield Theorem~\ref{thm:uniq_ns}.

If the sequence $\{ v_n \}_{n \ge 1}$ is an unconditional basis in ${\mathcal H}_p$, one can solve Inverse Problem~\ref{ip:main} by the following algorithm.

\begin{alg} \label{alg:1}
Suppose that the integer $p$, the entire functions $f_1(\la)$ and $f_2(\la)$, and the subspectrum $\{ \la_n \}_{n \ge 1}$ are given. We have to find $\sigma(x)$, $p_1(\la)$, and $p_2(\la)$.

\begin{enumerate}
\item Construct the functions $v(t, \la)$ and $w(\la)$ by either \eqref{defv1}, \eqref{defw1} or \eqref{defv2}, \eqref{defw2}, depending on the parity of $p$.
\item Construct the sequences $\{ v_n \}_{n \ge 1}$ and $\{ w_n \}_{n \ge 1}$ by \eqref{defvw}.
\item Find the biorthonormal sequence $\{ v_n^* \}_{n \ge 1}$ to $\{ v_n \}_{n \ge 1}$ in ${\mathcal H}_p$, that is,
$$
(v_n, v_k^*) = \begin{cases}
                    1, \quad n= k, \\
                    0, \quad n \ne k.
               \end{cases}
$$
\item Find the element $u \in {\mathcal H}_p$ satisfying \eqref{scal} by the formula
$$
u = \sum_{n = 1}^{\infty} \overline{w_n} v_n^*.
$$
\item Using the entries of $u$ (see \eqref{defu}), find $\Delta_0(\la)$ and $\Delta_1(\la)$ by the formulas of Lemma~\ref{lem:reprsdeltas}, and then find $M(\la) = \dfrac{\Delta_0(\la)}{\Delta_1(\la)}$.
\item Recover $\sigma(x)$ and the polynomials $p_1(\la)$, $p_2(\la)$ from the Weyl function $M(\la)$ using Algorithm 6.1 from \cite{ChitBond}.
\end{enumerate}
\end{alg}

\section{Sufficient conditions for completeness and basicity} \label{sec:suff}

In this section, we prove Theorem~\ref{thm:basis}, which presents simple sufficient conditions on the subspectrum $\{ \la_n \}_{n \ge 1}$ for completeness and basicity of the sequence $\{ v_n \}_{n \ge 1}$. We show that, instead of the sequence $\{ v_n \}_{n \ge 1}$, which is constructed by using arbitrary analytic functions $f_1(\la)$ and $f_2(\la)$, we can equivalently consider the sequence $\{ g_n \}_{n \ge 1}$, which is constructed by the characteristic functions $\Delta_0(\la)$ and $\Delta_1(\la)$ possessing some nice properties. Furthermore, we prove that the completeness and the unconditional basicity of $\{ g_n \}_{n \ge 1}$ follow from the analogous properties of the sequence $\{ \sin \rho_n t\}_{n > p}$ in $L_2(0, 2\pi)$.

Throughout this section, we suppose that $\{ \la_n \}_{n \ge 1} \in \mathcal S$ and $f_1(\la_n) \neq 0$ or $f_2(\la_n) \neq 0$ for every $n \ge 1$, as in Theorem~\ref{thm:basis}. For definiteness, we provide all the proofs for odd $p$. The case of even $p$ can be considered analogously.

Let us introduce the following vector function:
\begin{align}\notag
g(t, \la) = [&\rho^{p}\sin\rho t \Delta_0(\la), -\rho^{p-1}\cos\rho t\Delta_1(\la),\\ \label{g_def}
&\Delta_0(\la), \Delta_1(\la), \dots, \Delta_1(\la)\rho^{p-3},\Delta_0(\la)\rho^{p-1}],
\end{align}
and define $g_n(t) = g(t, \la_n)$, $n \ge 1$.

\begin{lem}\label{lem:compl_v}
	The system $\{v_n\}_{n \ge 1}$ is complete in the Hilbert space $\mathcal{H}_p$ if and only if the system $\{g_n\}_{n \ge 1}$ is complete in $\mathcal{H}_p$.
\end{lem}

\begin{proof}

Let $h = [\overline{H_1(t)}, \overline{H_2(t)}, \overline{h_1}, \dots, \overline{h_p}]$ be such an element of $\mathcal{H}_p$ that $(h, v_n)_{\mathcal{H}_p} = 0$, $n \ge 1$. 

Define the function
\begin{align} \nonumber
V(\la) = (h(t), v(t,\la))_{\mathcal H_p} &
= \int_0^\pi \Big(H_1(t)f_1(\la)\rho^p\sin\rho t + H_2(t)f_2(\la)\rho^{p-1}\cos\rho t \Big) \, dt \\ \label{defV} & +h_p f_1(\la)\rho^{p-1} + \sum\limits_{k=0}^{N_1-1}(h_{2k+1} f_1(\la)+h_{2k+2}f_2(\la))\rho^{2k},
\end{align}
meanwhile $V(\la_n) = 0$, $n \ge 1$.

Similarly, introduce the function
\begin{align} \nonumber
Z(\la) = (h(t), g(t, \la))_{\mathcal H_p} & = 
\int_0^\pi \Big(H_1(t)\Delta_0(\la)\rho^p\sin\rho t - H_2(t)\Delta_1(\la)\rho^{p-1}\cos\rho t \Big) \, dt 
\\ \label{defZ} & + h_p\Delta_0(\la)\rho^{p-1} + \sum\limits_{k=0}^{N_1-1}(h_{2k+1}\Delta_0(\la)-h_{2k+2}\Delta_1(\la))\rho^{2k}.
\end{align}

Let us show that $Z(\la_n) = 0$, $n \ge 1$. Recall that $f_1(\la_n) \neq 0$ or $f_2(\la_n) \neq 0$ for every $n \ge 1$. For definiteness, consider the case $f_1(\la_n) \ne 0$. Let us show that $\Delta_0(\la_n) \neq 0$. Indeed, if $\Delta_0(\la_n) = 0$, then from \eqref{charfun}, we obtain that $f_1(\la_n)\Delta_1(\la_n)=0$. Using Proposition~\ref{prop:zeros}, we conclude that $f_1(\la_n)=0$. This contradiction yields that $\Delta_0(\la_n)\neq0$. Consequently, from the relation
$$
\Delta(\la_n) = f_1(\la_n)\Delta_1(\la_n) + f_2(\la_n)\Delta_0(\la_n) = 0,
$$
we get
\begin{equation} \label{relf2}
f_2(\la_n) = -\dfrac{\Delta_1(\la_n)}{\Delta_0(\la_n)}f_1(\la_n).
\end{equation}

Substituting $\la = \la_n$ and \eqref{relf2} into \eqref{defV}, we obtain
$$
V(\la_n)= \frac{f_1(\la_n)}{\Delta_0(\la_n)}Z(\la_n) = 0, 
$$
and, consequently, $Z(\la_n) = 0$. In the case $f_2(\la_n) \ne 0$, we get $\Delta_1(\la_n) \ne 0$ and proceed similarly.

Thus, we arrive at the relation $(h, g_n)_{\mathcal{H}_p} = 0$, $n \ge 1$. In the opposite direction, the reasoning is similar. Hence, the completeness of $\{ v_n \}_{n \ge 1}$ is equivalent to the completeness of $\{ g_n \}_{n \ge 1}$.
\end{proof}

\begin{lem} \label{lem:suff}
Suppose that the system $\{\sin\rho_n t\}_{n > p}$ is complete in $L_2(0, 2\pi)$. Then the system $\{g_n\}_{n \ge 1}$ is complete in ${\mathcal {H}}_p$.
\end{lem}

\begin{proof}
Let $h = [\overline{H_1(t)}, \overline{H_2(t)}, \overline{h_1}, \dots, \overline{h_p}]$ be an element orthogonal to all $g_n$, $n \ge 1$, in $\mathcal H_p$. Then $Z(\la_n) = 0$, $n \ge 1$, where $Z(\la)$ is defined in \eqref{defZ}.

Due to the representations \eqref{Delta1} and \eqref{Delta0} for $N_1 \ge N_2$, we get the estimates:
\begin{gather}\label{est_Delta0}
\Delta_0(\la) = \rho^{p-1}\cos\rho\pi + \rho^{p-1}\varkappa_\pi(\rho) + O(|\rho|^{p-2}e^{\pi|\tau|}),\\ \label{est_Delta1}
\Delta_1(\la) = -\rho^{p}\sin\rho\pi + \rho^{p}\varkappa_\pi(\rho) + O(|\rho|^{p-1}e^{\pi|\tau|}),
\end{gather}
where $\tau := \Im \rho$.

Taking \eqref{defZ}, \eqref{est_Delta0}, and \eqref{est_Delta1} into account, we obtain
$$
Z(\la) = \la^p \Big( Z_1(\la) + O(|\rho|^{-2} e^{2\pi|\tau|}) \Big), \quad |\la| \to \infty,
$$
where $\rho Z_1(\rho^2)$ is an odd Paley-Wiener function. Then
$$
Z(\la) = \la^{p}\Bigg( \int_0^{2\pi} {z(t)\dfrac{\sin\rho t}{\rho}}\, dt + O(|\rho|^{-2}e^{2\pi|\tau|}) \Bigg),
$$
where $z(t) \in L_2(0, 2\pi)$. Exclude the first $p$ zeros and define the following function:
$$
R(\la) = \dfrac{Z(\la)}{\prod\limits_{n=1}^{p}(\la-\la_n)}.
$$
It follows from the definition, that $R(\la_n) = 0$, $n > p$. As $\rho R(\rho^2)$ is an odd function and $\rho R(\rho^2) \in L_2(\mathbb{R})$, then we obtain
$$
R(\la) = \int_0^{2\pi} {r(t)\dfrac{\sin\rho t}{\rho}}\, dt,
$$
where $r(t) \in L_2(0, 2\pi)$.

Consequently, if the system $\left\{ \sin\rho_n t \right\}_{n > p}$ is complete in $L_2(0, 2\pi)$, then $r(t) = 0$ a.e. on $(0, 2\pi)$, $R(\la) \equiv 0$, and $Z(\la) \equiv 0$.

Next, let $\{ \theta_n \}_{n \ge 1}$ be the zeros of $\Delta_1(\la)$. Then
$$
Z(\theta_n) = \theta_n^{N_1} \int_0^{\pi} {H_1(t)\sqrt{\theta_n}\sin(\sqrt{\theta_n} t) \Delta_0(\theta_n)}\, dt + \Delta_0(\theta_n)\sum\limits_{k=0}^{N_1}h_{2k+1}\theta_n^k = 0.
$$
In view of Proposition~\ref{prop:zeros}, we have $\Delta_0(\theta_n) \neq 0$. Hence 
$$
\theta_n^{N_1} \int_0^{\pi} {H_1(t)\sqrt{\theta_n}\sin(\sqrt{\theta_n}t)}\, dt + \sum\limits_{k=0}^{N_1}h_{2k+1}\theta_n^k = 0.
$$

In other words, $H(\theta_n) = 0$, $n \ge 1$, where
$$
H(\la) := \la^{N_1}\int_0^{\pi} {H_1(t)\sqrt{\la}\sin(\sqrt{\la} t)}\, dt + \sum\limits_{k=0}^{N_1}h_1^k\la^k.
$$
Next, let us exclude the first $N_1+1$ zeros and define the function
$$
F(\la) := \dfrac{H(\la)}{\prod\limits_{n=1}^{N_1+1}(\la-\theta_n)}.
$$

It can be shown, that $\rho F(\rho^2)$ is an odd Paley-Wiener function. So, 
$$
\rho F(\rho^2) = \int_0^\pi f(t)\sin\rho t \, dt,
$$
where $f(t) \in L_2(0, \pi)$, and, consequently,
$$
\int_0^\pi f(t)\dfrac{\sin\sqrt{\theta_n}t}{\sqrt{\theta_n}} \, dt = 0, \quad n > N_1 + 1.
$$

Using methods from \cite{FrYu} it can be shown that
\begin{gather}\label{as_prove}
\sqrt{\theta_n} = n - N_1 - 1 + \varkappa_n, \quad n \ge 1, \quad \{\varkappa_n\} \in l_2.
\end{gather}

For simplicity, assume that the values $\{\theta_n\}_{n > N_1+1}$ are distinct. The opposite case requires minor changes. 
As due the \eqref{as_prove} the system $\left\{ \frac{\sin\sqrt{\theta_n}t}{\sqrt{\theta_n}} \right\}_{n > N_1 + 1}$ is complete in $L_2(0, \pi)$, then $f(t) = 0$ a.e. on $(0,\pi)$, and so $H(\la) \equiv 0$. Hence $H_1(t), h_{2k+1} \equiv 0$, $k = \overline{0, N_1}$. Consequently, $H_2(t), h_{2k+2}$, $k = \overline{0, N_1-1}$, are also identically equal zero. It means that $(h, v_n)_{\mathcal{H}_p} = 0$ for $n \ge 1$ implies $h = 0$ in $\mathcal H_p$. This yields the claim.
\end{proof}

Lemmas~\ref{lem:compl_v} and~\ref{lem:suff} together imply the first part of Theorem~\ref{thm:basis}. 

Proceed to the proof of the second part. In order to prove the unconditional-basis property for the system $\{ v_n \}_{n \ge 1}$, we consider several auxiliary systems. Introduce the system $\{\xi_n\}_{n>p}$ of the elements $\xi_n = [\sin\rho_nt\cos\rho_n\pi,  -\cos\rho_nt\sin\rho_n\pi]$ in $\mathcal{X} := L_2(0, \pi)\oplus L_2(0, \pi)$ and the system $\{ g_n^0\}_{n \ge 1} \subset \mathcal H_p$ of the following elements:
\begin{align*}
g_n^0 = \Bigg[&\cos\rho_n\pi\sin\rho_nt, -\sin\rho_n\pi\cos\rho_nt, \dfrac{\cos\rho_n\pi}{\rho_n^p}, -\dfrac{\sin\rho_n\pi}{\rho_n^{p-1}}, \dots, -\dfrac{\sin\rho_n\pi}{\rho_n^2}, \dfrac{\cos\rho_n\pi}{\rho_n}\Bigg]
\end{align*}

Starting from the Riesz-basis property of the system $\{ \sin \rho_n t \}_{n > p}$, we investigate the basicity of the systems step-by-step in the following order:
$$
\{ \sin \rho_n t \}_{n > p} \: \Rightarrow \: \{ \xi_n \}_{n > p} \: \Rightarrow \:
\{ g_n^0 \}_{n \ge 1} \: \Rightarrow \: \{ g_n \}_{n \ge 1} \: \Rightarrow \: \{ v_n\}_{n \ge 1}.
$$ 

\begin{lem}\label{lem:xi_complete}
If the system $\{\sin\rho_nt\}_{n > p}$ is complete in $L_2(0, 2\pi)$, then the system $\{\xi_n\}_{n>p}$ is complete in $\mathcal{X}$.
\end{lem}

\begin{proof}

Let $h = [\overline{H_1(t)}, \overline{H_2(t)}]$ be an orthogonal element to all $\xi_n$, $n > p$, in  $\mathcal X$. Then
$$
\int_0^{\pi} \Big( \cos\rho_n\pi\sin\rho_ntH_1(t)-\sin\rho_n\pi\cos\rho_ntH_2(t)\Big) \, dt = 0,
$$
or in another form
$$
\dfrac{1}{2}\int_0^{\pi} \Big( (H_1(t) - H_2(t))\sin\rho_n(t+\pi)+(H_1(t) + H_2(t))\sin\rho_n(t-\pi)\Big) \, dt = 0.
$$

Consider the first term of the integrand:
$$
\int_0^{\pi} (H_1(t) - H_2(t))\sin\rho_n(t+\pi) \, dt = \int_{\pi}^{2\pi} (H_1(u-\pi) - H_2(u-\pi))\sin\rho_nu \, du = \int_{0}^{2\pi} F_1(u)\sin\rho_nu \, du,
$$
where
\begin{gather*}
F_1(u) = \left\{ \begin{aligned} 
H_1(u-\pi) - H_2(u-\pi), \quad u \in [\pi, 2\pi], \\
0, \quad u \in [0, \pi].
\end{aligned} \right.
\end{gather*}

In the same way, we can get that
$$
\int_0^{\pi} (H_1(t) - H_2(t))\sin\rho_n(t+\pi) \, dt = \int_{0}^{2\pi} F_2(u)\sin\rho_nu \, du,
$$
where
\begin{gather*}
F_2(u) = \left\{ \begin{aligned} 
H_1(\pi-u) + H_2(\pi - u), \quad u \in [0, \pi], \\
0, \quad u \in [\pi, 2\pi].
\end{aligned} \right.
\end{gather*}

Consequently, 
$$
\int_{0}^{2\pi} (F_1(u) + F_2(u))\sin\rho_nu \, du = 0.
$$

The completeness of the system $\{\sin\rho_nt\}_{n > p}$ in $L_2(0, 2\pi)$ implies that $F_1(u) = -F_2(u)$ a.e. on $[0, 2\pi]$. So, we get that $H_1(\pi-u) + H_2(\pi-u) = 0$ for a.a. $u \in [0,\pi]$ and, consequently, $H_1(t) = -H_2(t)$ for a.a. $t \in [0, \pi]$. In the same way, one can show that $H_1(t) = H_2(t)$ a.e. on $[0, \pi]$. It follows from $ H_2(t) = -H_2(t) $ a.e. on $[0, \pi]$ that $H_2(t) = 0$ and $H_1(t) = 0$ a.e. on $[0,\pi]$.

Thus, we have proved that, if $(h, \xi_n)_{\mathcal{X}} = 0$ for all $n > p$, then $h = 0$ in $\mathcal{X}$. This yields the claim.
\end{proof}

\begin{lem}\label{lem:gn0_complete_minimal}
	Suppose that the system $\{\sin\rho_nt\}_{n > p}$ is complete in $L_2(0, 2\pi)$. Then, the system $\{g_n^0\}_{n \ge 1}$ is complete in $\mathcal{H}_p$. Moreover, if $\{\xi_n\}_{n > p}$ is a Riesz basis in $\mathcal{X}$, then the system $\{g_n^0\}_{n \ge 1}$ is minimal in $\mathcal{H}_p$.
\end{lem}

\begin{proof}
The completeness is proved in the same way as for the system $\{g_n\}_{n \ge 1}$ in Lemma~\ref{lem:suff}, so we omit it.
	
Let us exclude the element $g_1^0$ and consider the system $\{g_n^0\}_{n \ge 2}$. Let $\{\hat g_n^0\}_{n \ge 2}$ be introduced as follows:
\begin{equation} \label{defgh}
\hat g_n^0 = \begin{cases}
g_n^0 - \sum\limits_{k=p+1}^{\infty}c_k^ng_k^0, & n \le p, \\
g_n^0, & n > p,
\end{cases}
\end{equation}
where $c_k^n$ are expansion coefficients of the $\{\xi_k\}_{k=1}^{p}$ with respect to the system $\{\xi_k\}_{k > p}$. It means that
$$
\xi_k = \sum\limits_{n=p+1}^{\infty}c_k^n\xi_n, \quad k=\overline{1, p}.
$$
Note that 
\begin{equation} \label{formgh}
\hat g_n^0 = \begin{cases}
[0, 0, \tilde g_n^0], & n \le p, \\
[\xi_n, \tilde g_n^0], & n > p,
\end{cases}
\end{equation}
where $\{ \tilde g_n^0\}$ are some vectors of $\mathbb C^p$.

Suppose that the system $\{g_n^0\}_{n \ge 2}$ is complete in $\mathcal H_p$. Then, the system $\{\hat g_n^0\}_{n \ge 2}$ is also complete. Indeed, any element $g_n^0$ can be expressed as a linear combination of the elements of $\{\hat g_n^0\}_{n \ge 2}$ and vice versa.

Consider the expansion of an element $\eta = [0, 0, *, \dots, *] \in \mathcal{H}_p$ by the system $\{\hat g_n^0\}_{n \ge 2}$. It cannot include the elements $\{\hat g_n^0\}_{n > p}$: these elements contain $\xi_n$, but $\{\xi_n\}_{n>p}$ is a Riesz basis in $\mathcal{X}$, so the relation $\sum\limits_{n > p} a_n\xi_n = 0$ implies $a_n = 0$, $n > p$. Therefore, $\eta = \sum\limits_{n=2}^{p} s_n \hat g_n^0$. But the system of $\mathbb C^p$-vectors $\{ \tilde g_n^0 \}_{n = 2}^p$ is incomplete in $\mathbb C^p$.
Then, not all elements $\eta$ can be expanded by the system 
$\{\hat g_n^0\}_{n \ge 2}$. Hence, this system is not complete, which yields the claim.
\end{proof}

\begin{lem}\label{lem:basis}
If the numbers $\{ \rho_n \}_{n \ge 1}$ satisfy the conditions Theorem~\ref{thm:basis}	and $\{ \sin \rho_n t\}_{n > p}$ is a Riesz basis 
in $L_2(0, 2\pi)$, then $\{ g_n^0\}_{n \ge 1}$ is a Riesz basis in $\mathcal{H}_p$.
\end{lem}

\begin{proof}
The proof consists of the two steps:
\begin{enumerate}
\item If $\{\sin\rho_n t\}_{n > p}$ is a Riesz basis in $L_2(0, \pi)$, then $\{\xi_n\}_{n>p}$ is a Riesz basis in $\mathcal{X}$;
\item If $\{\sin\rho_n t\}_{n > p}$ is complete in $L_2(0, \pi)$ and $\{\xi_n\}_{n>p}$ is a Riesz basis in $\mathcal{X}$, then $\{ g_n^0\}_{n \ge 1}$ is a Riesz basis in $\mathcal{H}_p$.
\end{enumerate}

\textbf{Step 1.} According to Lemma~\ref{lem:xi_complete}, we get that $\{\xi_n\}_{n>p}$ is complete in $\mathcal{X}$. Therefore, for proving that that $\{\xi_n\}_{n>p}$ is a Riesz basis in $\mathcal{X}$, it is sufficient to show that there exist positive numbers $M_1$ and $M_2$ such that, for any sequence $\{b_n\}$, the following two-side inequality holds:
\begin{gather}\label{xi_ineq}
M_1 \sum\limits_{n=p+1}^{N_0}|b_n|^2 \le \Bigg\|\sum\limits_{n=p+1}^{N_0}b_n\xi_n\Bigg\|^2_{\mathcal{X}} \le M_2 \sum\limits_{n=p+1}^{N_0}|b_n|^2.
\end{gather}

As the system $\{\sin\rho_nt\}_{n > p}$ is a Riesz basis in $L_2(0, 2\pi)$, so the following inequality holds for some positive numbers $M_1$ and $M_2$ and any sequence $\{b_n\}$:
\begin{gather}\label{sin_ineq}
M_1 \sum\limits_{n=p+1}^{N_0}|b_n|^2 \le \Bigg\|\sum\limits_{n=p+1}^{N_0}b_n\sin\rho_nt\Bigg\|^2_{L_2(0, 2\pi)} \le M_2 \sum\limits_{n=p+1}^{N_0}|b_n|^2.
\end{gather}

Calculations show that
\begin{gather}\label{sin_scal}
\Bigg\|\sum\limits_{n=p+1}^{N_0}b_n\sin\rho_nt\Bigg\|^2_{L_2(0, 2\pi)} = \sum\limits_{n=p+1}^{N_0}\sum\limits_{k=p+1}^{N_0}b_nb_k(\sin\rho_nt, \sin\rho_kt)_{L_2(0, 2\pi)}, \\ \label{xi_scal}
\Bigg\|\sum\limits_{n=p+1}^{N_0}b_n\xi_n\Bigg\|^2_{\mathcal{X}} = \sum\limits_{n=p+1}^{N_0}\sum\limits_{k=p+1}^{N_0}b_nb_k(\xi_n, \xi_k)_{\mathcal{X}}.
\end{gather}

Calculating the scalar products, we pass to the equality
\begin{gather}\label{sin_eq_xi}
(\sin\rho_nt, \sin\rho_kt)_{L_2(0, 2\pi)} = 2(\xi_n, \xi_k)_{\mathcal{X}}.
\end{gather}

Consequently, substituting \eqref{sin_scal} and \eqref{sin_eq_xi} into \eqref{sin_ineq} and taking \eqref{xi_scal} into account, we arrive at \eqref{xi_ineq}.

\smallskip

\textbf{Step 2.} According to Lemma~\ref{lem:gn0_complete_minimal}, the system $\{ g_n^0\}_{n \ge 1}$ is complete in $\mathcal{H}_p$. Therefore, for proving that $\{ g_n^0\}_{n \ge 1}$ is a Riesz basis in $\mathcal{H}_p$, it is sufficient to show that 
\begin{gather}\label{gn0_ineq}
M_1 \sum\limits_{n=1}^{N_0}|b_n|^2 \le \Bigg\|\sum\limits_{n=1}^{N_0}b_ng_n^0\Bigg\|^2_{\mathcal{H}_p} \le M_2 \sum\limits_{n=1}^{N_0}|b_n|^2
\end{gather}
for some positive constants $M_1$ and $M_2$ and any sequence $\{ b_n \}$.

Let us introduce an additional system $\{u_n\}_{n \ge 1} \subset \mathcal H_p$ of the elements
\begin{gather*}
[0, 0, 1, 0, \dots, 0], [0, 0, 0, 1, \dots, 0], \dots, [0, 0, 0, 0, \dots, 1], \\
[\cos\rho_n\pi\sin\rho_nt, -\sin\rho_n\pi\cos\rho_nt, 0, 0, \dots, 0], \quad n > p,
\end{gather*}
and show that $\{ u_n \}_{n \ge 1}$ is a Riesz basis in $\mathcal{H}_p$. It follows from the completeness of the system $\{\xi_n\}_{n>p}$ in $\mathcal{X}$ that the system $\{u_n\}_{n \ge 1}$ is complete in $\mathcal{H}_p$.

Calculations show that
$$
\Bigg\|\sum\limits_{n=1}^{N_0}b_nu_n\Bigg\|^2_{\mathcal{H}_p} = \sum\limits_{n=1}^{p}|b_n|^2 + \Bigg\|\sum\limits_{n=p+1}^{N_0}b_n\xi_n\Bigg\|^2_{\mathcal{X}}.
$$

Consequently, adding $\sum\limits_{n=1}^{p}|b_n|^2$ to all the three parts of \eqref{xi_ineq}, we get that
$$
M_1 \sum\limits_{n=1}^{N_0}|b_n|^2 \le \Bigg\|\sum\limits_{n=1}^{N_0}b_nu_n\Bigg\|^2_{\mathcal{H}_p} \le M_2 \sum\limits_{n=1}^{N_0}|b_n|^2,
$$
so $\{u_n\}_{n \ge 1}$ is a Riesz basis in $\mathcal{H}_p$.

Next, consider an arbitrary element $f = \sum\limits_{k=1}^\infty b_kg_k^0$, where $\{b_k\} \in l_2$. It can be shown that $f \in \mathcal{H}_p$. So, we can expand it with respect to the Riesz basis $\{u_n\}_{n \ge 1}$:
$$
f = \sum\limits_{k=1}^\infty a_ku_k, \quad \{a_k\} \in l_2,
$$
or in element-wise form
\begin{gather}\notag
f = \Bigg[\sum\limits_{n=1}^\infty b_n\cos\rho_n\pi\sin\rho_nt, -\sum\limits_{n=1}^\infty b_n\sin\rho_n\pi\cos\rho_nt, \sum\limits_{n=1}^\infty b_n\dfrac{\cos\rho_n\pi}{\rho_n^p}, -\sum\limits_{n=1}^\infty b_n\dfrac{\sin\rho_n\pi}{\rho_n^{p-1}}, \dots, \\ \label{f_expand} -\sum\limits_{n=1}^\infty b_n\dfrac{\sin\rho_n\pi}{\rho_n^{2}}, 
 \sum\limits_{n=1}^\infty b_n\dfrac{\cos\rho_n\pi}{\rho_n} \Bigg]= \Bigg[\sum\limits_{n=1}^\infty a_n\cos\rho_n\pi\sin\rho_nt, -\sum\limits_{n=1}^\infty a_n\sin\rho_n\pi\cos\rho_nt, a_1, \dots, a_p\Bigg].
\end{gather}

Now, expand the set of the vector-functions $\{[\cos\rho_n\pi\sin\rho_nt, -\sin\rho_n\pi\cos\rho_nt]\}_{n=1}^p$ with respect to the Riesz basis $\{\xi_n\}_{n > p}$:

\begin{gather}\label{xi_dec}
[\cos\rho_k\pi\sin\rho_kt, -\sin\rho_k\pi\sin\rho_kt] = \sum\limits_{n=p+1}^{\infty}c_n^k\xi_n, \quad \{c_n^k\}\in l_2, \quad k=\overline{1, p}.
\end{gather}

Then from \eqref{f_expand} we get:
\begin{gather*}
\sum\limits_{n=1}^{\infty}b_n\cos\rho_n\pi\sin\rho_nt = \sum\limits_{n=p+1}^{\infty}\Big(b_n+\sum\limits_{k=1}^{p}b_kc_n^k\Big)\cos\rho_n\pi\sin\rho_nt = \sum\limits_{n=p+1}^{\infty}a_n\cos\rho_n\pi\sin\rho_nt.
\end{gather*}
Consequently,
\begin{gather}\label{a_from_b}
a_n = b_n+\sum\limits_{k=1}^{p}b_kc_n^k, \quad n > p.
\end{gather}

The relation \eqref{f_expand} yields
\begin{gather}\label{a_1_p}
a_1 = \sum\limits_{n=1}^{\infty} b_n\dfrac{\cos\rho_n\pi}{\rho_n^p}, 
\quad a_2 = -\sum\limits_{n=1}^\infty b_n\dfrac{\sin\rho_n\pi}{\rho_n^{p-1}}, \quad
\dots, \quad a_p = \sum\limits_{n=1}^{\infty} b_n\dfrac{\cos\rho_n\pi}{\rho_n},
\end{gather}
Using \eqref{a_from_b}, \eqref{a_1_p}, the property $\{ \la_n \}_{n \ge 1} \in \mathcal A$, and taking $\{a_n\}, \{b_n\}, \{c_n^k\} \in l_2$ for any $k=\overline{1, p}$ into account, we obtain the estimate
\begin{gather}\label{a_less_b}
\sum\limits_{n=1}^{\infty}|a_n|^2 \le C_1\sum\limits_{n=1}^{\infty}|b_n|^2,
\end{gather}
where the constant $C_1$ does not depend on $\{ a_n\}$ and $\{ b_n \}$.

Next, from \eqref{a_1_p}, using \eqref{a_from_b}, we get
\begin{gather*}
a_1 = \sum\limits_{n=1}^{\infty} b_n\dfrac{\cos\rho_n\pi}{\rho_n^p} = \sum\limits_{k=1}^{p} b_k\Bigg(\dfrac{\cos\rho_k\pi}{\rho_k^p} - \sum\limits_{n=p+1}c_n^k \dfrac{\cos\rho_n\pi}{\rho_n^p} \Bigg) + \sum\limits_{n=p+1}^{\infty} a_n\dfrac{\cos\rho_n\pi}{\rho_n^p}.
\end{gather*}
For $\{a_n\}_{n=2}^{p}$, we obtain equations of the same form. So, we pass to the following system of linear equations:
\begin{gather}\label{slau}
Ab=a,
\end{gather}
where
\begin{gather}\notag
A =
  \begin{bmatrix}
    \dfrac{\cos\rho_1\pi}{\rho_1^p} - \sum\limits_{n=p+1}^{\infty} c_n^1\dfrac{\cos\rho_n\pi}{\rho_n^p} & \cdots & \dfrac{\cos\rho_p\pi}{\rho_p^p} - \sum\limits_{n=p+1}^{\infty} c_n^p\dfrac{\cos\rho_n\pi}{\rho_n^p}\\
    \vdots & \ddots & \vdots \\
    \dfrac{\cos\rho_1\pi}{\rho_1} - \sum\limits_{n=p+1}^{\infty} c_n^1\dfrac{\cos\rho_n\pi}{\rho_n} & \cdots & \dfrac{\cos\rho_p\pi}{\rho_p} - \sum\limits_{n=p+1}^{\infty} c_n^p\dfrac{\cos\rho_n\pi}{\rho_n}
  \end{bmatrix}, \\ \notag
  b = [b_1, \dots, b_p]^T,\\ \notag
  a = \Bigg[a_1 - \sum\limits_{n=p+1}^{\infty} a_n \dfrac{\cos\rho_n\pi}{\rho_n^p}, \dots, a_p - \sum\limits_{n=p+1}^{\infty} a_n \dfrac{\cos\rho_n\pi}{\rho_n}\Bigg]^T,
\end{gather}
where the symbol $T$ means the matrix transpose.

Let us prove the unique solvability of the system \eqref{slau}. Note that the matrix $A^T$ consists of the rows $\tilde g_n^0$, $n = \overline{1,p}$, which were defined in \eqref{formgh}. Suppose that $|A| = 0$. Then, there exists a non-trivial linear combination of its rows that equals zero:
$$
\sum\limits_{k=1}^{p} a_k\tilde g_k^0 = 0.
$$

Due to \eqref{xi_dec}, we get that
$$
\sum\limits_{k=1}^{p} a_k\Big(\xi_k - \sum\limits_{n=p+1}^{\infty} c_n^k \xi_n\Big) = 0.
$$

Then, the analogous linear combination of $g_n^0$ has the following form:
$$
\sum\limits_{k=1}^{p} a_k\Big(g_k^0 - \sum\limits_{n=1}^{\infty} c_n^k g_n^0\Big) = \Bigg[\sum\limits_{k=1}^{p} a_k\Big(\xi_k - \sum\limits_{n=p+1}^{\infty} c_n^k \xi_n\Big), \sum\limits_{k=1}^{p} a_k\tilde g_k^0\Bigg] = 0.
$$

It means that the system $\{g_n^0\}_{n \ge 1}$ is not minimal in $\mathcal{H}_p$. But this conflicts with Lemma~\ref{lem:gn0_complete_minimal}.
Consequently, $|A| \neq 0$ and the system \eqref{slau} is uniquely solvable. Then, we find $b = A^{-1}a$ and show that
\begin{gather}\label{b_less_a}
\sum\limits_{n=1}^{\infty}|b_n|^2 \le C_2\sum\limits_{n=1}^{\infty}|a_n|^2.
\end{gather}

Using \eqref{a_less_b}, \eqref{b_less_a}, and \eqref{xi_ineq}, we get \eqref{gn0_ineq}. This yields the claim.
\end{proof}

To prove the second part of Theorem~\ref{thm:basis}, we need to pass from $\{g_n^0\}_{n \ge 1}$ to $\{ v_n \}_{n \ge 1}$.

\begin{lem}\label{gn0_to_v}
Suppose that the system $\{\sin\rho_n t\}_{n > p}$ is complete in $L_2(0, 2\pi)$, the system $\{g_n^0\}_{n \ge 1}$ is a Riesz basis in $\mathcal{H}_p$, and $\{\la_n\}_{n \ge 1} \in \mathcal{A}$. Then $\{ v_n \}_{n \ge 1}$ is also a Riesz basis in $\mathcal{H}_p$.
\end{lem}

\begin{proof}
Due to \eqref{g_def} and the estimates \eqref{est_Delta0}-\eqref{est_Delta1}, we get
\begin{align*}
g_n(t) = &\rho_n^{2p-1}g_n^0\\
&+ \rho_n^{2p-1}\Big[ \sin\rho t \varkappa_\pi(\rho_n), -\cos\rho t \varkappa_\pi(\rho_n), \frac{\varkappa_\pi(\rho_n)}{\rho_n^p}, -\frac{\varkappa_\pi(\rho_n)}{\rho_n^{p-1}}, ..., , -\frac{\varkappa_\pi(\rho_n)}{\rho_n^2}, \frac{\varkappa_\pi(\rho_n)}{\rho_n} \Big]\\
& + O(|\rho_n|^{2p-2}e^{2\pi|\tau|}).
\end{align*}
Therefore, according to the asymptotical condition $\{ \la_n \} \in \mathcal A$, we obtain 
$$
\sum_{n = 1}^{\infty} \| \rho_n^{-(2p-1)} g_n(t) - g_n^0(t) \|^2_{\mathcal H_p} < \infty.
$$

Thus, the sequence $\left\{ \rho_n^{-(2p-1)} g_n(t) \right\}_{n \ge 1}$ is $l_2$-close to the Riesz basis $\{ g_n^0\}_{n \ge 1}$. Taking the completeness of $\{ g_n \}_{n \ge 1}$ into account (see Lemma~\ref{lem:suff}), 
we conclude that $\{g_n\}_{n \ge 1}$ is an unconditional basis in $\mathcal{H}_p$.

Due to Lemma~\ref{lem:compl_v}, we get that the system $\{ v_n \}_{n \ge 1}$ is complete in $\mathcal{H}_p$. Simple calculations show that $g_n = k_nv_n$, $n \ge 1$, where $\{k_n\}_{n \ge 1}$ are some non-zero numbers. Together with the completeness of $\{ v_n \}_{n \ge 1}$ in $\mathcal{H}_p$, this yields the claim.
\end{proof}

\section{Local solvability and stability} \label{sec:stab}

The main goal of this section is to prove Theorem~\ref{thm:stab}. 
Let the problem $\tilde L = L(\tilde \sigma, \tilde p_1, \tilde p_2, f_1, f_2)$ and its simple subspectrum $\{ \tilde \la_n \}_{n \ge 1}$ fulfill the conditions of Theorem~\ref{thm:stab}. Let the functions $v(t, \la)$ and $w(\la)$ be constructed by the formulas \eqref{defv1}, \eqref{defw1} or \eqref{defv2}, \eqref{defw2} depending on parity of the integer $p$. Here we provide the proof only for odd $p$, but in the opposite case all steps are the similar. Throughout this section, we denote by the same symbol $C$ various constants depending only on $\tilde L$ and $\{ \tilde \la_n \}_{n \ge 1}$.

\begin{lem}\label{lem:cond}
Under Conditions 1--3 of Theorem~\ref{thm:stab}, there exists a positive constant $a_4$ such that 
\begin{gather} \label{ineq_nm}
|\tilde\rho_n - \tilde\rho_m| \ge a_4, \quad n \neq m. 
\end{gather}
Moreover, the following estimates hold for $|\rho - \tilde \rho_n| \le a_2$:
\begin{gather}\label{est_v_le}
\|v(t, \rho^2)\|_{\mathcal{H}_p} \le C|\tilde\rho_n|^{\alpha_n+p}, \\ \label{est_w_le}
|w(\rho^2)| \le C|\tilde\rho_n|^{\alpha_n+p}.
\end{gather}
\end{lem}

\begin{proof}
First of all, Condition 1 implies that $\Big\{\dfrac{\tilde v_n}{\|\tilde v_n\|_{\mathcal{H}_p}}\Big\}_{n \ge 1}$ is a Riesz basis in $\mathcal{H}_p$.

Let us prove the estimate $|\tilde\rho_n - \tilde\rho_m| \ge a_4$, $n \ne m$, by contradiction. Suppose that there exist sequences of indices $\{i_k\}$ and $\{j_k\}$, $i_k, j_k \ge 1$, such that
\begin{gather}\label{ijk_est}
|\tilde\rho_{i_k} - \tilde\rho_{j_k}| = o(1), \quad k \to \infty.
\end{gather}

Consider the norm of $v(t, \rho^2)$ in $\mathcal{H}_p$:
\begin{align*}
\|v(t, \rho^2)\|_{\mathcal{H}_p} = & \Bigg( \int_0^\pi {\big( |f_1(\la)^2||\rho|^{2p}\sin\rho t \sin\overline{\rho}t + |f_2(\la)^2||\rho|^{2p-2}\cos\rho t \cos\overline{\rho}t \big)}dt \\ & +
|f_1(\la)|^2\sum\limits_{n=0}^{N_1}|\rho|^{4n} + |f_2(\la)|^2\sum\limits_{n=0}^{N_1-1}|\rho|^{4n} \Bigg)^{\frac{1}{2}} \\ \le & |f_1(\la)||\rho|^p\sqrt{\pi\max_{t \in [0, \pi]}(\sin\rho t \sin\overline{\rho}t)} 
+ |f_2(\la)||\rho|^{p-1}\sqrt{\pi\max_{t \in [0, \pi]}(\cos\rho t \cos\overline{\rho}t)} \\ & + |f_1(\la)|(1+|\rho|^2+ \dots +|\rho|^{p-1}) + |f_2(\la)|(1+|\rho|^2+\dots+|\rho|^{p-3}).
\end{align*}

Under the condition $|\rho-\tilde\rho_n| \le a_2$ and Condition~3, we get \eqref{est_v_le}. In the same way, we obtain the estimate for $w(\rho^2)$:
\begin{gather*}
|w(\rho^2)| = |f_1(\la)\rho^p\sin\rho\pi - f_2(\la)\rho^{p-1}\cos\rho\pi| \le |f_1(\la)||\rho|^p|\sin\rho\pi| + |f_2(\la)||\rho|^{p-1}|\cos\rho\pi|.
\end{gather*}
Due to $|\rho-\tilde\rho_n| \le a_2$ and Condition~3, we arrive at \eqref{est_w_le}.

According to the Schwartz lemma for the vector function $(v(t, \tilde\rho_n^2) - v(t, \tilde\rho^2))$, we get
$$
\|v(t, \tilde\rho_n^2) - v(t, \rho^2)\|_{\mathcal{H}_p} \le C|\tilde\rho_n|^{\alpha_n+p}|\rho-\tilde\rho_n|, \quad |\rho-\tilde\rho_n| \le a_2,
$$
and, due to \eqref{ijk_est},
\begin{gather}\label{v_diff_est}
\|v(t, \tilde\rho_{i_k}^2) - v(t, \tilde\rho_{j_k}^2)\|_{\mathcal{H}_p} = o(|\tilde\rho_{i_k}|^{\alpha_{i_k}+p}), \quad k \to \infty.
\end{gather}

Using Condition~2, we can get that
\begin{gather}\label{v_est}
\|v(t, \tilde\rho_n^2)\|_{\mathcal{H}_p} \ge C|\tilde\rho_n|^{\alpha_n + p}.
\end{gather}

Now, using \eqref{v_diff_est} and \eqref{v_est}, we pass to the estimate
$$
\dfrac{\tilde v_{i_k}}{\|\tilde v_{i_k}\|_{\mathcal{H}_p}} - \dfrac{\tilde v_{i_k}}{\|\tilde v_{i_k}\|_{\mathcal{H}_p}} \to 0, \quad k\to\infty,
$$
that contradicts to the Riesz basicity of the system $\Big\{\dfrac{\tilde v_n}{\|\tilde v_n\|_{\mathcal{H}_p}}\Big\}_{n \ge 1}$. This yields the claim.
\end{proof}

Now we are ready to prove Theorem~\ref{thm:stab}.

\begin{proof}[Proof of Theorem~\ref{thm:stab}]

Let $\{ \la_n \}_{n \ge 1}$ be a perturbed subspectrum of the problem $\tilde L$
such that \eqref{est_omega} is fulfilled for sufficiently small $\varepsilon > 0$. Then, it follows from \eqref{est_omega} and \eqref{ineq_nm} that $\{ \la_n \}_{n \ge 1} \in \mathcal{S}$.

Define the vector-functions $\{v_n\}_{n \ge 1}$ and the numbers $\{w_n\}_{n \ge 1}$ as follows:
\begin{gather} \label{defvn}
v_n(t) = v(t, \la_n), \quad w_n = w(\la_n).
\end{gather}
Consider the system of equations
\begin{gather}\label{stab_system}
(u, v_n)_{\mathcal{H}_p} = w_n, \quad n \ge 1,
\end{gather}
with respect to the unknown vector function $u \in \mathcal H_p$.

Using the estimates \eqref{est_v_le}--\eqref{est_w_le} and the Schwartz lemma, we obtain for sufficiently small $\varepsilon$ that
\begin{gather}\label{stab_est3}
\|v_n - \tilde v_n\|_{\mathcal{H}_p} \le C|\tilde\rho_n|^{\alpha_n+p}|\rho_n - \tilde\rho_n|, \\ \label{stab_est4}
|w_n - \tilde w_n| \le C|\tilde\rho_n|^{\alpha_n+p}|\rho_n - \tilde\rho_n|.
\end{gather}

So, using \eqref{v_est}, \eqref{stab_est3}, and \eqref{stab_est4}, we get the important estimates:
\begin{gather}\label{stab_est5}
\Bigg( \sum\limits_{n=1}^{\infty} \|\tilde v_n\|_{\mathcal{H}_p}^{-2} \|\tilde v_n - v_n\|_{\mathcal{H}_p}^{2} \Bigg)^{\frac{1}{2}} \le C\Omega, \\ \label{stab_est6}
\Bigg( \sum\limits_{n=1}^{\infty} \|\tilde v_n\|_{\mathcal{H}_p}^{-2} |\tilde w_n - w_n|^{2} \Bigg)^{\frac{1}{2}} \le C\Omega,
\end{gather}
where $\Omega = \Bigg(\sum\limits_{n=1}^{\infty}|\tilde\rho_n - \rho_n|^2\Bigg)^{\frac{1}{2}}$.

Then, for sufficiently small $\varepsilon > 0$, the system $\Bigg\{ \dfrac{v_n}{\|v_n\|_{\mathcal{H}_p}} \Bigg\}_{n \ge 1}$is a Riesz basis in $\mathcal H_p$. Consequently, there exists a unique element $u \in \mathcal{H}_p$ satisfying \eqref{stab_system}.

Next, consider the following values:
\begin{gather*}
Z := \Bigg( \sum\limits_{n=1}^{\infty} \Bigg\| \dfrac{\tilde v_n}{\| \tilde v_n \|_{\mathcal{H}_p}} - \dfrac{v_n}{\| v_n \|_{\mathcal{H}_p}} \Bigg\|^2 \Bigg)^{\frac{1}{2}}, \\
T := \Bigg( \sum\limits_{n=1}^{\infty} \Bigg| \dfrac{\tilde w_n}{\| \tilde v_n \|_{\mathcal{H}_p}} - \dfrac{w_n}{\| v_n \|_{\mathcal{H}_p}} \Bigg|^2 \Bigg)^{\frac{1}{2}}.
\end{gather*}

Using \eqref{est_v_le}, \eqref{est_w_le}, \eqref{v_est}, \eqref{stab_est3}, and \eqref{stab_est4}, we can get
\begin{gather}\label{ZT_est}
Z \le C\Omega, \quad T \le C\Omega.
\end{gather}

Then, due to Lemma 5 from \cite{Bond2018}, using the estimates \eqref{ZT_est}, we get that
$$
\|u - \tilde u\|_{\mathcal{H}_p} \le C(Z+T) \le C\Omega.
$$

Let the data $\{\mathcal J(t), \mathcal G(t), A_1, \dots, A_p \}$ be extracted from the vector $u$ due to \eqref{defu}. Then, for sufficiently small $\varepsilon > 0$ ($\Omega \le \varepsilon$),
Proposition~\ref{cauchy_thm} implies that there exist a complex-valued function $\sigma(x) \in L_2(0, \pi)$ and polynomials $(p_1(\la), p_2(\la)) \in \mathcal R_p$ related to the generalized Cauchy data $\{\mathcal J(t), \mathcal G(t), A_1, \dots, A_p \}$. So, we can build the problem $L = L(\sigma, p_1, p_2, f_1, f_2)$. Moreover, due the \eqref{Delta1}-\eqref{Delta0} and Proposition~\ref{cauchy_thm} estimates \eqref{coef_estimates} are also correct.

As the last step, we need to show that $\{\la_n\}_{n \ge 1}$ is a subspectrum of the problem $L$.
Due the \eqref{defu}, \eqref{defv1}, \eqref{defw1} and \eqref{defvn} equation \eqref{stab_system} can be rewritten in the form
\begin{gather} \notag
\int_0^\pi {\Big( \mathcal{J}(t)f_1(\la_n)\rho_n^p\sin\rho_nt + \mathcal{G}(t)f_2(\la_n)\rho_n^{p-1}\cos\rho_nt \Big) \, dt}  \\ \label{stab_sys_exp}
+ f_1(\la_n)\sum\limits_{k=0}^{N_1}A_{2k+1}\rho^{2k} + f_2(\la_n)\sum\limits_{k=0}^{N_1-1}A_{2k+2}\rho^{2k} = f_1(\la_n)\rho_n^p\sin\rho_n\pi - f_2(\la_n)\rho_n^{p-1}\cos\rho_n\pi.
\end{gather}

Define the functions $\Delta_0(\la)$ and $\Delta_1(\la)$ by \eqref{Delta1}-\eqref{Delta0}. Then equation \eqref{stab_sys_exp} takes the form
$$
\big( f_1(\la)\Delta_1(\la) + f_2(\la)\Delta_0(\la) \big)\Big|_{\la=\la_n} = 0.
$$

So, $\{\la_n\}_{n \ge 1}$ are zeros of the function $\Delta(\la) = f_1(\la)\Delta_1(\la) + f_2(\la)\Delta_0(\la) $. As $\{ \mathcal{J}(t), \mathcal{G}(t), A_1, ..., A_{p}\}$ are generalized Cauchy data of the problem $L$ and $f_1(\la) \equiv \tilde f_1(\la)$, $f_2(\la) \equiv \tilde f_2(\la)$, then $\Delta(\la)$ is the characteristic function of the problem $L$. Consequently, $\{\la_n\}_{n \ge 1}$ is a subspectrum of the problem $L$.
\end{proof}

\section{Hochstadt-Lieberman-type problem} \label{sec:h-l}

In this section, we apply our results to the Hochstadt-Lieberman-type problem with polynomial dependence on the spectral parameter in the boundary conditions.

Introduce the following boundary value problem $\mathcal{L} = \mathcal{L}(\sigma, p_1, p_2, r_1, r_2)$:
\begin{gather}\label{h-l:eq}
-y''(x)+q(x)y(x)=\la y(x), \quad x\in(0, 2\pi), \quad q  = \sigma' \in W_2^{-1}(0, 2\pi)\\ \notag
p_1(\la)y^{[1]}(0)+p_2(\la)y(0) = 0, \quad  (p_1(\la), p_2(\la)) \in \mathcal{R}_p, \\ \notag
r_1(\la)y^{[1]}(2\pi)+r_2(\la)y(2\pi) = 0, \quad (r_1(\la), r_2(\la)) \in \mathcal{R}_r.
\end{gather}

For definiteness, suppose that the degrees $p$ and $r$ are odd. Other cases require minor technical changes.

Let us denote the spectrum of the problem $\mathcal{L}$ by $\{\eta_n\}_{n \ge 1}$ and consider the following Hochstadt-Lieberman-type problem:

\begin{ip}\label{ip:h-l}
Suppose that the integer $p$, the polynomials $r_1(\la)$, $r_2(\la)$ and the function $\sigma(x)$ for $x \in (\pi, 2\pi)$ are known a priori. Given a spectrum $\{\eta_n\}$ of the problem $\mathcal{L}$, find $\sigma(x)$ for $x \in (0, \pi)$ and the polynomials $p_1(\la)$, $p_2(\la)$.
\end{ip}

Define the functions $\varphi(x, \la)$ and $\psi(x, \la)$ as the solutions of the equation~\eqref{h-l:eq}, satisfying the initial conditions
$$
\varphi(0, \la) = p_1(\la), \quad \varphi^{[1]}(0, \la) = -p_2(\la), \quad \psi(2\pi, \la) = r_1(\la), \quad \psi^{[1]}(2\pi, \la) = -r_2(\la).
$$

It can be shown similarly to \cite{Chit} that the eigenvalues of the problem $\mathcal{L}$ coincide with zeros of the following characteristic function:
$$
\Delta(\la) = -\psi(\pi, \la)\Delta_1(\la) + \psi^{[1]}(\pi, \la)\Delta_0(\la).
$$
This means that $\{\eta_n\}_{n \ge 1}$ coincide with eigenvalues of the problem $L = L(\sigma, p_1, p_2, f_1, f_2)$ with
\begin{gather}\label{def_f1-2}
f_1(\la) = -\psi(\pi, \la), \quad f_2(\la) = \psi^{[1]}(\pi, \la).
\end{gather}

Next, we suppose that $\{\eta_n\}_{n \ge 1} \in \mathcal{S}$ and $\eta_n \neq 0$, $n \ge 1$. In general, a finite number of eigenvalues can be multiple, but reasoning is more technically complicated.

\begin{lem}\label{lem:h-l_est}
Let $f_1(\la)$ and $f_2(\la)$ be the entire functions, defined by formulas \eqref{def_f1-2}, and $\{\eta_n\}_{n \ge 1} \in \mathcal{S}$ be the spectrum of boundary value problem $\mathcal{L}$. Then
\begin{enumerate}
\item The system $\{ \sin\sqrt{\eta_n}t \}_{n \ge p+r+1}$ is a Riesz basis in $L_2(0, 2\pi)$;
\item $\{\eta_n\}_{n \ge 1} \in \mathcal{A}$;
\item $f_1(\eta_n) \neq 0$ or $f_2(\eta_n) \neq 0$ for any $n \ge 1$;
\item The estimates \eqref{estf} hold for $\la_n = \eta_n$ and $\alpha_n = r-1$ for $n \ge 1$.
\end{enumerate}
\end{lem}

\begin{proof}

It can be shown that, in this case, the spectrum fulfills the following asymptotics \cite{Chit}:
$$
\sqrt{\eta_n} = \dfrac{n}{2} - \dfrac{p+r}{4} + \varkappa_n, \quad n \ge 1,
$$
where $\{\varkappa_n\} \in l_2$. This asymptotics implies the assertions 1 and 2 of the lemma.

Assertion 3 also holds, because the functions $\psi(\pi, \la)$ and $\psi^{[1]}(\pi, \la)$ do not have common zeros. Indeed, if there exist $\eta \in \mathbb{C}$, such that $\psi(\pi, \eta) = \psi^{[1]}(\pi, \eta) = 0$, then $\psi(x, \eta)$ is the trivial solution of equation \eqref{h-l:eq}, which is impossible.

It can be shown, that function $\psi(x, \la)$ can be represented as
$$
\psi(x, \la) = r_1(\la)C(2\pi-x, \la) + r_2(\la)S(2\pi-x, \la),
$$
where $S(x, \la)$ and $C(x, \la)$ are the solution of equation~\eqref{h-l:eq} satisfying the initial conditions \eqref{initphi}.

Then, the following estimates for functions $f_1(\la)$ and $f_2(\la)$ are correct:
\begin{gather*}
f_1(\la) = -\rho^{r-1}\cos\rho\pi-\rho^{r-1}\varkappa_\pi(\rho) + O(|\rho|^{r-3}e^{\pi|\tau|}),\\
f_2(\la) = \rho^{r}\sin\rho\pi+\rho^{r}\varkappa_\pi(\rho) + O(|\rho|^{r-2}e^{\pi|\tau|}),
\end{gather*}
where $\rho^2 = \lambda$. Taking $|\rho-\sqrt{\eta_n}| \le C$ in account, we get
$$
|f_1(\rho^2)| \le C|\sqrt{\eta_n}|^{r-1}, \quad |f_2(\rho^2)| \le C|\sqrt{\eta_n}|^{r}.
$$

Next, we can get the last estimate, using simple calculations:
\begin{align*}
|f_1(\eta_n)|^2 + |\eta_n|^{-1}|f_2(\eta_n)|^2 & \ge |f_1^2(\eta_n) + \eta_n^{-1}f_2^2(\eta_n)| \\
& = |\eta_n^{r-1} + \eta_n^{r-1} \varkappa_{2\pi}(\sqrt{\eta_n}) + O(|\rho|^{2r-4})e^{2\pi|\tau|}| \ge C|\eta_n|^{r-1}.
\end{align*}

So, the estimates \eqref{estf} are valid with $\alpha_n = r-1$. This yields the claim.
\end{proof}

Due to Lemma~\ref{lem:h-l_est}, we can apply Theorems~\ref{thm:basis} and~\ref{thm:stab} to Inverse Problem~\ref{ip:h-l} and so obtain the following corollaries.

\begin{cor} \label{cor:uniq}
Let $\{\eta_n\}_{n \ge 1}$ and $\{\tilde \eta_n\}_{n \ge 1}$ be the spectra of the problems $\mathcal{L}=\mathcal{L}(\sigma, p_1, p_2, r_1, r_2)$ and $\tilde {\mathcal{L}}=\mathcal{L}(\tilde{\sigma}, \tilde p_1, \tilde p_2, \tilde r_1, \tilde r_2)$, respectively. Assume that $p = \tilde p$, $r_j(\la) = \tilde r_j(\la)$, $j = 1, 2$, and $\sigma(x) = \tilde \sigma(x)$ a.e. on $(\pi, 2\pi)$. Moreover, suppose that $r \ge p$ and $\eta_n = \tilde \eta_n$, $n \ge \frac{r-p}{2} + 1$. Then, $\sigma(x) = \tilde \sigma(x)$ a.e. on $(0, \pi)$ and $p_j(\la) = \tilde p_j(\la)$, $j=1, 2$.
\end{cor}

\begin{cor} \label{cor:stab}
For any $(\tilde p_1(\la), \tilde p_2(\la)) \in \mathcal R_p$, $\tilde \sigma(x) \in L_2(0, \pi)$ there exists $\varepsilon > 0$ such that for any complex numbers $\{\eta_n\}_{n \ge \frac{r-p}{2}+1}$ satisfying the condition
$$
\Omega = \Bigg(\sum\limits_{n=\frac{r-p}{2}+1}^{\infty}|\sqrt{\tilde{\eta_n}} - \sqrt{\eta_n}|^2\Bigg)^{\frac{1}{2}} \le \varepsilon,
$$
there exist complex-valued function $\sigma(x) \in L_2(0, \pi)$ and polynomials $(p_1(\la), p_2(\la)) \in \mathcal R_p$, such that $\{\eta_n\}_{n \ge \frac{r-p}{2}+1}$ is a subspectrum of the problem $\mathcal{L} = \mathcal{L}(\sigma, p_1, p_2, r_1, r_2)$. Moreover
$$
\|\sigma(x) - \tilde\sigma(x)\|_{L_2(0, \pi)} \le C\Omega, \quad |a_j - \tilde a_j| \le C\Omega, \quad |b_j - \tilde b_j| \le C\Omega,
$$
where $C$ depends only on $r_1(\la)$, $r_2(\la)$ and $\{\tilde\eta_n\}_{n \ge \frac{r-p}{2}+1}$
\end{cor}

Thus, due to Corollary~\ref{cor:uniq}, any $\frac{r - p}{2}$ eigenvalues can be excluded from the spectrum $\{ \eta_n \}_{n \ge 1}$ in order to keep the uniqueness for solution of Inverse Problem~\ref{ip:h-l}. This coincides with the result of \cite[Theorem 3]{Chit} for regular potentials. The exclusion of more eigenvalues implies the non-uniqueness. This follows from the local solvability for the Hochstadt-Lieberman-type inverse problem given by Corollary~\ref{cor:stab}.

\medskip

\textbf{Funding}: This work was supported by Grant 24-71-10003 of the Russian Science Foundation, https://rscf.ru/en/project/24-71-10003/.

\medskip

\textbf{Competing interests}: The paper has no conflict of interests.



\medskip

\noindent Egor Evgenevich Chitorkin \\
1. Institute of IT and Cybernetics, Samara National Research University,\\ 
Moskovskoye Shosse 34, Samara 443086, Russia. \\
2. Department of Mechanics and Mathematics, Saratov State University, \\
Astrakhanskaya 83, Saratov 410012, Russia. \\
e-mail: {\it chitorkin.ee@ssau.ru} 

\medskip

\noindent Natalia Pavlovna Bondarenko \\
1. Department of Applied Mathematics and Physics, Samara National Research University,\\
Moskovskoye Shosse 34, Samara 443086, Russia.\\
2. S.M. Nikolskii Mathematical Institute, Peoples' Friendship University of Russia (RUDN University), 6 Miklukho-Maklaya Street, Moscow, 117198, Russia.\\
3. Moscow Center of Fundamental and Applied Mathematics, Lomonosov Moscow State University, Moscow 119991, Russia.\\
e-mail: {\it bondarenkonp@sgu.ru}\\
\end{document}